\documentclass{article}
\usepackage{verbatim,amsmath,hyperref,amssymb}
\usepackage{amsthm}
\usepackage{graphicx}
\usepackage{mathpazo}
\usepackage{multirow}
\usepackage{booktabs}
\usepackage{caption}
\usepackage{subcaption}
\usepackage{color}

\newcounter{enucount}

\newtheorem{prp}[enucount]{Proposition}

\newcommand{\ones}{\textbf{1}}

\newcommand{\conv}{\textrm{conv}}

\newcommand{\st}{\textrm{s.t.}}
\newcommand{\fctp}{\textrm{(FCTP)}}
\newcommand{\ip}{\textrm{(IP)}}
\newcommand{\ipz}{\textrm{(IP+}z\textrm{)}}
\newcommand{\ipl}[1]{\textrm{(IP}_{#1}\textrm)}

\title{Fixed-charge transportation problems on trees}
\author{Gustavo Angulo\textsuperscript{*}\footnote{Center for Operations Research and Econometric, Universit\'e catholique de Louvain}\textsuperscript{\dag}\footnote{Departamento de Ingenier\'ia Industrial y de Sistemas, Pontificia Universidad Cat\'olica de Chile}\quad Mathieu Van Vyve\textsuperscript{*}\\
gustavo.angulo@uclouvain.be \quad mathieu.vanvyve@uclouvain.be}

\begin{document}
\maketitle

\begin{abstract}
We consider a class of fixed-charge transportation problems over graphs. We show that this problem is strongly NP-hard, but solvable in pseudo-polynomial time over trees using dynamic programming. We also show that the LP formulation associated to the dynamic program can be obtained from extended formulations of single-node flow polytopes. Given these results, we present a unary expansion-based formulation for general graphs that is computationally advantageous when compared to a standard formulation, even if its LP relaxation is not stronger.\\


\end{abstract}

\section{Introduction}

Let $n$ and $m$ be positive integers, and let $N:=\{1,\ldots,n\}$ and $M:=\{1,\ldots,m\}$. 
We refer to $i\in N$ and $j\in M$ as suppliers and customers, respectively. Let $c_i\geq 0$ be the available supply at node $i$ and let $d_j\geq 0$ be the demand at node $j$. We assume that $c$ and $d$ are integer-valued.
Also, let $q_{ij}$ be the fixed cost for sending a positive amount of flow from $i$ to $j$, and let $p_{ij}$ be the per unit flow cost (or revenue). The fixed-charge transportation problem $\fctp$ asks for a set of flows from suppliers to customers satisfying capacity limits at nodes and arcs such that the sum of fixed and variable costs is the least possible. A natural integer programming formulation is given by
\begin{eqnarray}
\nonumber \ipl{nm}&\min& p^\top x + q^\top y\\
\label{fctp_sup}&\st&\sum_{j=1}^m x_{ij}\leq c_i\quad  i\in N\\
\label{fctp_dem}&&\sum_{i=1}^n x_{ij}\leq d_j \quad  j\in M\\
\label{fctp_ub}&&y_{ij}\leq x_{ij}\leq \min\{c_i,d_j\}y_{ij} \quad  i\in N,\ j\in M\\
\label{fctp_bin}&&y_{ij}\in\{0,1\} \quad  i\in N,\ j\in M,
\end{eqnarray}
where $y_{ij}$ indicates whether or not there is a positive flow from $i$ to $j$, and $x_{ij}$ is the amount of flow. In view of (\ref{fctp_sup}) and (\ref{fctp_dem}), the roles of suppliers and customers are symmetric in the above formulation. However, the analysis we present here can be modified to fit models where the inequalities in (\ref{fctp_sup}) and (\ref{fctp_dem}) are reversed or set to equalities.

$\fctp$ generalizes the single-node flow problem and is therefore at least as hard (weakly NP-hard). Flow cover \cite{padberg1985valid,van1987solving} and lifted flow cover inequalities \cite{gu1999lifted} have been shown to be very effective. For the equality case, \cite{agarwal2012fixed} derives inequalities and facets for the projection on the $y$-space, and \cite{roberti2014fixed} introduces a reformulation with exponentially many variables which is solved with column generation. When the underlying graph is a path, \cite{van2013fixed} presents a class of path-modular inequalities and shows that they suffice to describe the convex hull of feasible solutions. For the single-node flow problems, \cite{alidaee2005note} presents an improved dynamic program.

For notational convenience, we formulate $\fctp$ as follows.
Let $G=(V,E)$ be a graph. Let $b_i\geq 0$ be the capacity of node $i\in V$. For each arc $(i,j)=(j,i)\in E$, we set $a_{ij}:=\min\{b_i,b_j\}$ as the capacity of the arc. Given cost vectors $p$ and $q$, consider
\begin{eqnarray}
\nonumber \ip&\min& p^\top x+q^\top y\\
\label{ipg_cap}&\st& \sum_{j\in V:\ (i,j)\in E}x_{ij} \leq b_i \quad i\in V\\
\label{ipg_ub}&& 0 \leq x_{ij}\leq a_{ij}y_{ij} \quad (i,j)\in E\\
\label{ipg_ybin}&&y_{ij}\in\{0,1\} \quad (i,j)\in E.
\end{eqnarray}
When $G$ is isomorphic to $K_{nm}$, the complete bipartite graph with $n$ and $m$ nodes on each side of the partition, we recover $\ipl{nm}$. 

Let $S$ be the set defined by (\ref{ipg_cap})--(\ref{ipg_ybin}), and let $P$ be its linear relaxation obtained by replacing (\ref{ipg_ybin}) with $0\leq y_{ij}\leq 1$. 
We denote $\conv(X)$ the convex hull of a set $X$ of real vectors. Our main object of study is $\conv(S)$.

Consider now the following unary-expansion based formulation
\begin{eqnarray}
\nonumber \ipz&\min& p^\top x+q^\top y\\
\nonumber &\st& \eqref{ipg_cap}, \eqref{ipg_ub}, \eqref{ipg_ybin} \\
\label{uip_xz}&& \sum_{l=0}^{a_{ij}} l\ z_{ijl} = x_{ij} \quad (i,j)\in E\\
\label{uip_yz}&& \sum_{l=1}^{a_{ij}} z_{ijl} \leq y_{ij} \quad (i,j)\in E\\
\label{uip_z}&& \sum_{l=0}^{a_{ij}} z_{ijl} = 1 \quad (i,j)\in E\\
\label{bnd_z}&& z_{ijl}\in\{0,1\} \quad (i,j)\in E, 0 \leq l \leq a_{ij},
\end{eqnarray}
where the intended meaning is that $z_{ijl}=1$ if $x_{ij}=l$ and $0$ otherwise. As the next easy proposition shows, $\ipz$ is \emph{not} stronger than $\ip$. 
\begin{prp}
The projection of the linear relaxation of $\ipz$ onto $(x,y)$ is equal to $P$.
\end{prp}
\begin{proof}
We just need to prove the reverse inclusion. Given $(x,y)\in P$, it suffices to exhibit a vector $0\leq z\leq \ones$ such that (\ref{uip_xz})--(\ref{uip_z}) are satisfied. For each arc $(i,j)$, let $z_{ija_{ij}}=x_{ij}/a_{ij}$, $z_{ij0}=1-z_{ija_{ij}}$ and $z_{ijl}=0$ for $0<l<a_{ij}$. Verifying that the constraints are satisfied is straightforward.
\end{proof}
However, we will see that modern IP solvers are much more effective at tightening $\ipz$ than $\ip$. 

This paper can be read as a theoretical explanation of this empirical finding. More specifically our contributions are the following. In Section~\ref{dp} we prove that  $\fctp$ is strongly NP-hard in general, but is pseudo-polynomially solvable when $G$ is a tree. Based on this last result, we show in Section~\ref{formulations} that single-node tightening of $\ipz$ is actually sufficient to describe $\conv(S)$ when $G$ is a tree. Finally we empirically demonstrate in Section~\ref{comp} the substantial performance improvements obtained when using $\ipz$ to solve with CPLEX randomly generated instances of $\fctp$, when numbers $b_i$ are not too large integers.

\section{Complexity Results} \label{dp}

\subsection{General graphs}
\begin{prp}
$\fctp$ is strongly NP-hard.
\end{prp} 
\begin{proof}
We give a reduction of the strongly NP-complete 3-Partition to $\fctp$ \cite{Garey:1979:CIG:578533}. In 3-Partition we are given $3n$ nonnegative integers $a_1,\hdots,a_{3n}$ that sum up to $nb$ and are strictly between $b/4$ and $b/2$. The problem is to decide whether they can be partitioned into $n$ groups that each sum up to $b$. Consider an instance of $\fctp$ with $n$ suppliers with capacity $b$ each, $3n$ clients with demands $a_1,\hdots,a_{3n}$, $p_{ij}=-2$ variable cost and unit fixed cost $q_{ij}=1$ for all pairs $(i,j)$. 

The optimal solution to this instance of $\fctp$ satisfies $\sum_{j} x_{ij}=b$ for all $i$ because the fixed cost of opening any arc is smaller than the benefit of sending one unit of flow along this arc. Therefore the problem reduces to using as few arcs as possible to saturate all nodes. Because each client has to be linked to at least one supplier, any solution has cost at least $-2nb+3n$.

If the optimal solution value is exactly $-2nb+3n$ and therefore each client is fully served by one supplier only, then this solution yields the desired partition and the answer to 3-Partition is yes. If the answer of 3-Partition is yes, then the (obvious) assignment of clients to suppliers according to the 3-Partition solution has indeed cost $-2nb+3n$.
\end{proof}

\subsection{Trees}
We assume in this subsection that $G$ is a tree. Let node 1 be the (arbitrarily chosen) root of $G$. For node $i$, let $p(i)$ denote its parent node (we set $p(1)=0$) and let $\{f(i),\ldots,l(i)\}$ denote its children, which we assume are numbered consecutively (as with breadth-first search). In presenting a dynamic programming formulation, it is understood that $p(j)=i$ for any arc $(i,j)\in E$ and thus $i<j$.

Let $(i,j)\in E$. For $0\leq l\leq a_{ij}$, let $\beta_{ijl}$ denote the optimal value of the problem restricted to the subtree induced by $i$, $j$, and its descendants, under the condition that $x_{ij}=l$ (see Fig.~\ref{fig:dp2}). Similarly, for $0\leq k\leq b_i$, let $\alpha_{ijk}$ denote the optimal value of the problem on the subtree induced by $i, f(i),\ldots,j$, and their descendants, under the condition that $\sum_{f(i)\leq j'\leq j}x_{ij'}=k$ (see Fig.~\ref{fig:dp3}). 
Then we have that the optimal value of $\ip$ is given by $\beta_{010}$ such that
\begin{equation}\label{dp_alpha}
\alpha_{ijk}= \left\{
\begin{array}{rl}
\displaystyle \beta_{ijk} & j=f(i)\\
\displaystyle \min_{0\leq k-k'\leq a_{ij}} \left\{\alpha_{i(j-1)k'} + \beta_{ij(k-k')}\right\} & j>f(i)
\end{array}\right. \quad (i,j)\in E,\ 0\leq k\leq b_i
\end{equation}
\begin{equation}\label{dp_beta}
\beta_{ijl}= \left\{
\begin{array}{rl}
\displaystyle c_{ijl} & j\ \textrm{is a leaf node}\\
\displaystyle \min_{0\leq k\leq b_j-l} \left\{\alpha_{jl(j)k}\right\} + c_{ijl} & j\ \textrm{is a nonleaf node}
\end{array}\right. \quad (i,j)\in E,\ 0\leq l\leq a_{ij}
\end{equation}
\begin{equation}\label{dp_ojb}
\beta_{010} = \min_{0\leq k\leq b_1}\left\{\alpha_{1l(1)k}\right\}
\end{equation}

where, for $(i,j)\in E$, we set $c_{ij0} = 0$ and $c_{ijl} = lp_{ij}+q_{ij}$ for $l\geq 1$.

\begin{figure}[h]
 \begin{center}
 \begin{subfigure}[b]{0.45\textwidth}
   \includegraphics[scale=0.3,keepaspectratio=true]{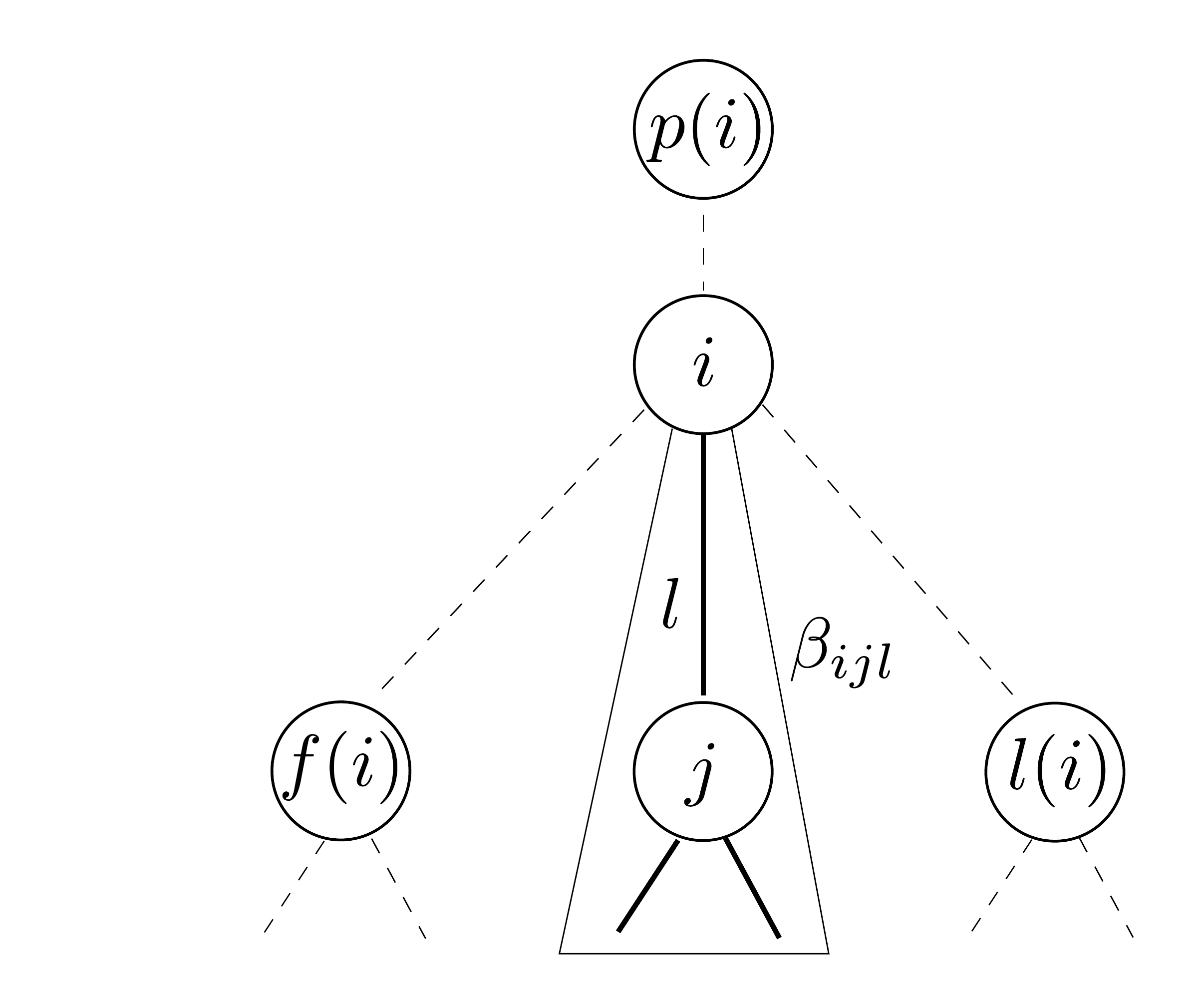}
   \caption{Restricted problem defining $\beta_{ijl}$.}
   \label{fig:dp2}
 \end{subfigure}
 \begin{subfigure}[b]{0.45\textwidth}
   \includegraphics[scale=0.3,keepaspectratio=true]{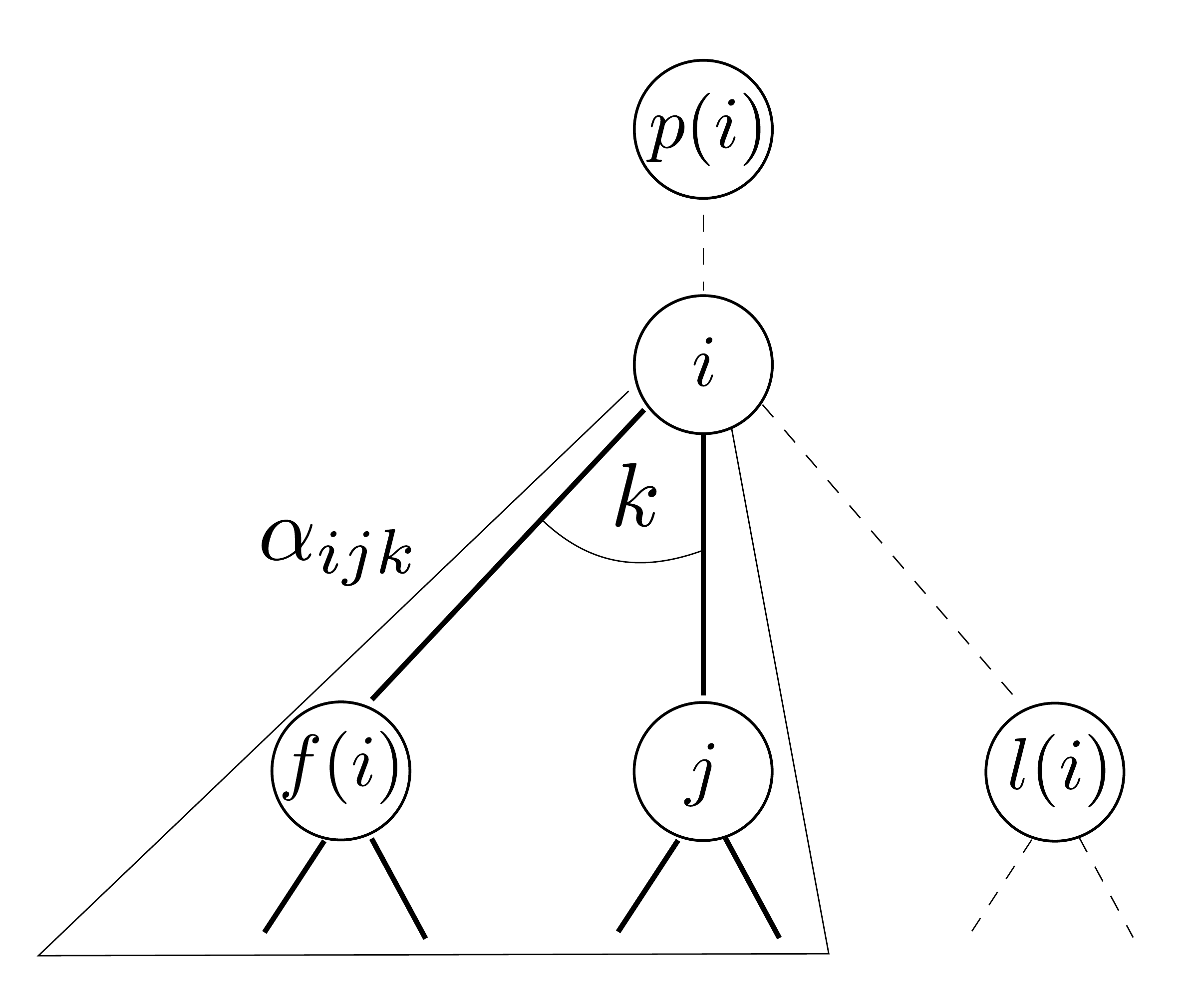}
   \caption{Restricted problem defining $\alpha_{ijk}$.}
   \label{fig:dp3}
 \end{subfigure}
\end{center}
\caption{Representation of dynamic program on a tree.}
\label{fig:dp}
\end{figure}

The key property exploited for this construction is the fact that for a tree, if the flow on an edge is fixed, the problem decomposes into two subproblems of exactly the same type. 

\section{Extended Formulations}\label{formulations}

In LP form, (\ref{dp_alpha})--(\ref{dp_ojb}) leads to
\begin{eqnarray*}
\max\ \beta_{010}\\
\st\ \alpha_{ijk}\leq \beta_{ijk} & (i,j)\in E,\ j=f(i),\ 0\leq k\leq b_i & (u_{ij0k})\\
\alpha_{ijk}\leq \alpha_{i(j-1)k'} + \beta_{ij(k-k')} & (i,j)\in E,\ j>f(i),\ 0\leq k'\leq k\leq b_i:\ k-k'\leq a_{ij} & (u_{ijk'k})\\
\beta_{ijl}\leq c_{ijl} & (i,j)\in E,\ 0\leq l\leq a_{ij},\ j\ \textrm{is a leaf node} & (v_{ijl0})\\
\beta_{ijl}\leq \alpha_{jl(j)k} + c_{ijl} & (i,j)\in E,\ 0\leq l\leq a_{ij},\ 0\leq k\leq b_j-l,\ j\ \textrm{is a nonleaf node} & (v_{ijlk})\\
\beta_{010}\leq \alpha_{1l(1)k} & 0\leq k\leq b_1 & (v_{010k}).
\end{eqnarray*}

Taking the dual we obtain
\begin{eqnarray}
\label{dual_obj} \min\ \sum_{(i,j)\in E}\sum_{k,l}c_{ijl}v_{ijlk}\\
\label{dual_flowuv} \st\ \sum_{0\leq k-k'\leq a_{ij}}u_{ijk'k} = 
\left\{\begin{array}{ll}
\displaystyle \sum_{0\leq k'-k\leq a_{i(j+1)}}u_{i(j+1)kk'} & f(i)\leq j < l(i)\\        
\displaystyle \sum_{0\leq l\leq b_i-k}v_{p(i)ilk} & j= l(i)\\
       \end{array}\right. & (i,j)\in E,\ 0\leq k\leq b_i \\
\label{dual_revuv} \sum_{0\leq k\leq b_j-l} v_{ijlk} = \sum_{k-k'=l}u_{ijk'k} & (i,j)\in E,\ 0\leq l\leq a_{ij}\\
\label{dual_unituv} \sum_{0\leq k\leq b_1}v_{010k} = 1 \\
\label{dual_u}u\geq 0\\
\label{dual_v}v\geq 0.
\end{eqnarray}

Note that the left-hand side of (\ref{dual_flowuv}) reduces to $u_{ij0k}$ if $j=f(i)$. Similarly, the left-hand side of (\ref{dual_revuv}) reduces to $v_{ijl0}$ if $j$ is a leaf, and its right-hand side equals to $u_{ij0l}$ if $j=f(i)$.

From (\ref{dual_obj}) and the definition of $c_{ijl}$, we have
$$\sum_{(i,j)\in E}\sum_{k,l}c_{ijl}v_{ijlk}=\sum_{(i,j)\in E}\sum_{l>0}(lp_{ij}+q_{ij})\sum_{k}v_{ijlk}
=\sum_{(i,j)\in E}p_{ij}\sum_{l}l\sum_{k}v_{ijlk}+\sum_{(i,j)\in E}q_{ij}\sum_{l>0}\sum_{k}v_{ijlk}.$$
Thus using the mappings $x_{ij} = \sum_{l}l\sum_{k}v_{ijlk}$ and $y_{ij}=\sum_{l>0}\sum_{k}v_{ijlk}$, (\ref{dual_flowuv})--(\ref{dual_v}) becomes an extended formulation for $\conv(S)$ of pseudo-polynomial size. Let us call it $Q_{DP}$.

On the other hand, we can derive another formulation for $S$, of pseudo-polynomial size as well, by writting an extended formulation for the convex hull of the single-node flow problem at each node of $G$ and adding an appropriate set of linking constraints. 
More precisely, for each $i\in V$, we consider a vector $f_i$ which indicates how flows from $i$ are assigned to $p(i)$ and $f(i),\ldots,l(i)$. 
Let $f_{ijk'k}$ be a binary variable taking the value 1 if and only if $\sum_{f(i)\leq j'< j}x_{ij'}=k'$ and $x_{ij}=k-k'$. If $j=f(i)$, then only  variables $f_{ij0k}$ are defined. Similarly, let $f_{ip(i)k'k}$ be a binary variable taking the value 1 if and only if $\sum_{f(i)\leq j'\leq l(i)}x_{ij'}=k'$ and $x_{ip(i)}=k-k'$. If $i$ is a leaf node, then only variables $f_{ip(i)0k}$ are defined, and if $i=1$, then we set $a_{10}=0$ and only variables $f_{10kk}$ are defined. 
Thus we have that $f=(f_1,\ldots,f_{|V|})$ must satisfy
\begin{eqnarray}
\label{flowf} \sum_{0\leq k-k'\leq a_{ij}} f_{ijk'k} = 
\left\{\begin{array}{ll}
\displaystyle \sum_{0\leq k'-k\leq a_{i(j+1)}} f_{i(j+1)kk'} & f(i)\leq j < l(i)\\        
\displaystyle \sum_{0\leq k'-k\leq a_{ip(i)}} f_{ip(i)kk'} & j= l(i)\\
       \end{array}\right. & (i,j)\in E,\ 0\leq k\leq b_i \\
\label{unitf} \sum_{0\leq k-k'\leq a_{ip(i)}}f_{ip(i)k'k} = 1 &  i\in V\\
\label{linkf} \sum_{k-k'=l}f_{ijk'k} = \sum_{k-k'=l}f_{jik'k} & (i,j)\in E,\ 0\leq l\leq a_{ij}\\
\label{f0} f\geq 0.
\end{eqnarray}
Since in this case we use the mappings $x_{ij}=\sum_{k'\leq k}(k-k')f_{ijk'k}$ and $y_{ij}=\sum_{k'<k}f_{ijk'k}$, constraints (\ref{linkf}) ensure that the flows from $i$ to $j$ and from $j$ to $i$ are consistent.
Let $Q_{SN}$ be the polyhedron defined by (\ref{flowf})--(\ref{f0}). By construction, $\conv(S)$ is contained in the projection of $Q_{SN}$ onto the $(x,y)$-space. We show that equality holds.

\begin{prp}\label{localgolbal}
$Q_{SN}$ is an extended formulation for $\conv(S)$.
\end{prp}
\begin{proof}
Since $Q_{DP}$ is an extended formulation for $\conv(S)$, it suffices to show that there exists an injective linear mapping $\pi:Q_{SN} \longrightarrow Q_{DP}$ such that for any $f\in Q_{SN}$, both $f$ and $\pi(f)$ have the same projection onto the $(x,y)$-space. Let $\pi(f)=(u,v)$ be defined by $u_{ijk'k}=f_{ijk'k}$ and $v_{ijlk}=f_{jik(k+l)}$ for $i<j$. 

We first show that $(u,v)\in Q_{DP}$. Clearly, (\ref{dual_u}) and (\ref{dual_v}) are implied by (\ref{f0}). Since $f$ satisfies (\ref{flowf}), then $(u,v)$ satisfies (\ref{dual_flowuv}). Constraint (\ref{dual_revuv}) is implied by (\ref{linkf}). 
Taking (\ref{unitf}) for $i=1$ yields
$$1=\sum_{0 \leq k\leq b_1}f_{1p(1)kk} = \sum_{0 \leq k\leq b_1}f_{10kk}=\sum_{0 \leq k\leq b_1}v_{010k},$$
implying (\ref{dual_unituv}). Therefore, $(u,v)\in Q_{DP}$ and $\pi$ is well-defined. Seeing that $\pi$ is injective is immediate from its definition. Verifying that the projections of $f$ and $\pi(f)$ coincide follows from (\ref{linkf}) and a simple change of variables.
\end{proof}

Note that \eqref{linkf} can be equivalently rewritten as:
\begin{eqnarray}
\label{linkf1} \sum_{k-k'=l}f_{ijk'k} = z_{ijl} & (i,j)\in E,\ 0\leq l\leq a_{ij}\\
\label{linkf2} \sum_{k-k'=l}f_{jik'k} = z_{ijl} & (i,j)\in E,\ 0\leq l\leq a_{ij},
\end{eqnarray}
where, as in $\ipz$, $z_{ijl}=1$ if $x_{ij}=l$ and $0$ otherwise. 
It should then be clear that (i) formulation  \eqref{flowf}--\eqref{unitf}, \eqref{f0}--\eqref{linkf2} is an exact extended formulation when $G$ is a tree, and (ii) that this exact formulation is obtained by convexifying (in the $(x,y,z)$-space of variables) each single-node relaxation of $\ipz$ separately.

Note that this is very much in the spirit of Bodur et al. \cite{bodur2015cutting}. They show that the LP relaxation of any mixed-binary linear program admits an extended formulation, not stronger than the original formulation, but whose split closure is integral. However, the proof is existential only in the sense that it requires a complete enumeration of all extreme points of the LP relaxation, which in general is an extremely demanding task. Here, for the specific case of $\fctp$, we exhibit an extended formulation, not stronger than the orginal one, but whose "single-node closure" is integral. Note that this is not the case for the standard formulation $\ip$: convexifying each single-node relaxation of $\ip$ separately (in the $(x,y)$-space of variables) does not yield an integral polyhedron, even when $G$ is a tree.

When $G$ is not a tree, Proposition~\ref{localgolbal} still suggests a strong (albeit not exact) formulation for $\fctp$: write the network flow-based extended formulation at each node of $\ipz$. The resulting formulation will be at least as strong as the intersection of all tree-relaxations in the original variable space.
However, the number of variables grows quadratically with the size of the entries of $b$, which is impractical except for very small values. 

Fortunately, modern IP solvers have very effective built-in cut separation routines for single-node relaxations of the type found in $\ipz$. Therefore instead of explicitely tightening these single-node relaxations in the model, a practical approach that we explore in the next section is to just give formulation $\ipz$ to the solver and rely on its automatic tightening capabilities. The intended benefit is to reduce the number of variables to linear from quadratic in the size of $b$, but at the cost of a reduced strength of the dual bound.

\section{Computational Experiments}\label{comp}

We generated random instances as follows. We consider problems of the form $\ipl{nm}$ where (\ref{fctp_dem}) is set to equality. The number of suppliers and customer is the same, $n$, chosen from $\{20,30,40\}$. We also choose a number $B\in\{20,40,60\}$ and a factor $r\in\{0.90,0.95,1.00\}$ as the total demand to total supply ratio. Given $B$ and $r$, we sample each $c_i$ and $d_j$ independently from the uniform distribution on $\{1,\ldots,B\}$. Let $C$ and $D$ be the sampled total supply and total demand, respectively. If $D<\lceil rC \rceil$, then we iterate over $j$ and increase $d_j$ by one unit if $d_j<B$, one $j$ at a time, until  $D=\lceil rC \rceil$. Similarly, if $D>\lceil rC \rceil$, then we iterate over $i$ and increase $c_i$ by one unit if $c_i<B$, one $i$ at a time, until  $D=\lceil rC \rceil$. Fixed cost are sampled independently from the uniform distribution on $\{200,\ldots,800\}$, while variable costs are set to zero. 

We generated 10 instances for each combination of $n$, $B$, and $r$, which are solved using formulations $\ip$ and $\ipz$. For each formulation, we gather information at two stages: after the root node has been processed (\emph{Root}) and at the end of branch-and-cut (\emph{B\&C}). We report average results for \emph{Time} (seconds), \emph{Gap} (\%), and \emph{Nodes}. For $\ipz$, we also report the average percentage difference of the lower bound (\emph{$\Delta$LB}) and upper bound (\emph{$\Delta$UB}) compared to those obtained by $\ip$ at the root node and at end of branch-and-cut. We use CPLEX 12.6.1 as the solver, running on Linux with a single thread and a time limit of 3600 seconds. Tables \ref{t20}, \ref{t30}, and \ref{t40} summarize our results. 

\begin{table}[ht]
\small
\begin{center}
\begin{tabular}{ccrrrrrrrrrrrr}
\toprule
	&		&	\multicolumn{4}{c}{$\ip$}					&\multicolumn{8}{c}{$\ipz$}						\\
				\cmidrule(lr){3-6}							\cmidrule(lr){7-14}
	&		&	\multicolumn{1}{c}{Root}		&\multicolumn{3}{c}{B\&C}	&\multicolumn{3}{c}{Root}	&\multicolumn{5}{c}{B\&C}			\\
			\cmidrule(lr){3-3}		\cmidrule(lr){4-6}				\cmidrule(lr){7-9}		\cmidrule(lr){10-14}
$B$	&	$r$	&	Time	&	Gap	&	Time	&	Nodes	&	$\Delta$LB&	$\Delta$UB&	Time	&	$\Delta$LB&	$\Delta$UB&	Gap	&	Time	&	Nodes	\\
\midrule
\multirow{3}{*}{20}	&	0.90	&	1	&	0.00	&	10	&	3297	&	3.08	&	-1.65	&	1	&	0.00	&	0.00	&	0.00	&	1	&	1	\\
	&	0.95	&	1	&	0.00	&	28	&	8596	&	4.49	&	-1.86	&	3	&	0.00	&	0.00	&	0.00	&	3	&	11	\\
	&	1.00	&	1	&	0.00	&	528	&	163231	&	6.92	&	25.65	&	5	&	0.00	&	0.00	&	0.00	&	6	&	58	\\
\midrule
\multirow{3}{*}{40}	&	0.90	&	1	&	0.00	&	24	&	8424	&	3.87	&	-1.37	&	3	&	0.00	&	0.00	&	0.00	&	4	&	70	\\
	&	0.95	&	1	&	0.00	&	289	&	131431	&	5.32	&	-3.10	&	5	&	0.00	&	0.00	&	0.00	&	6	&	63	\\
	&	1.00	&	1	&	3.04	&	3251	&	809102	&	9.50	&	123.09	&	13	&	2.64	&	-0.52	&	0.00	&	135	&	1487	\\
\midrule
\multirow{3}{*}{60}	&	0.90	&	1	&	0.00	&	24	&	9227	&	3.74	&	-1.48	&	5	&	0.00	&	0.00	&	0.00	&	6	&	94	\\
	&	0.95	&	1	&	0.00	&	677	&	228024	&	5.74	&	50.50	&	10	&	0.00	&	0.00	&	0.00	&	32	&	592	\\
	&	1.00	&	1	&	6.16	&	3610	&	985554	&	10.60	&	192.02	&	19	&	5.27	&	0.84	&	1.94	&	2294	&	16412	\\
\bottomrule
\end{tabular}
\caption{Average results on $20\times 20$ instances.}
\label{t20}
\end{center}
\end{table}


\begin{table}[ht]
\small
\begin{center}
\begin{tabular}{ccrrrrrrrrrrrr}
\toprule
	&		&	\multicolumn{4}{c}{$\ip$}					&\multicolumn{8}{c}{$\ipz$}						\\
				\cmidrule(lr){3-6}							\cmidrule(lr){7-14}
	&		&	\multicolumn{1}{c}{Root}		&\multicolumn{3}{c}{B\&C}	&\multicolumn{3}{c}{Root}	&\multicolumn{5}{c}{B\&C}			\\
			\cmidrule(lr){3-3}		\cmidrule(lr){4-6}				\cmidrule(lr){7-9}		\cmidrule(lr){10-14}
$B$	&	$r$	&	Time	&	Gap	&	Time	&	Nodes	&	$\Delta$LB&	$\Delta$UB&	Time	&	$\Delta$LB&	$\Delta$UB&	Gap	&	Time	&	Nodes	\\
\midrule
\multirow{3}{*}{20}	&	0.90	&	1	&	0.00	&	167	&	29033	&	2.06	&	-1.95	&	4	&	0.00	&	0.00	&	0.00	&	4	&	22	\\
	&	0.95	&	2	&	0.17	&	853	&	114655	&	3.87	&	-4.89	&	7	&	0.17	&	0.00	&	0.00	&	8	&	43	\\
	&	1.00	&	3	&	2.31	&	2905	&	308104	&	5.49	&	69.52	&	15	&	2.16	&	-0.22	&	0.00	&	68	&	789	\\
\midrule
\multirow{3}{*}{40}	&	0.90	&	2	&	0.00	&	626	&	106839	&	3.58	&	-2.36	&	10	&	0.00	&	0.00	&	0.00	&	16	&	180	\\
	&	0.95	&	2	&	0.87	&	2419	&	329429	&	4.76	&	-1.68	&	17	&	0.86	&	-0.03	&	0.00	&	42	&	475	\\
	&	1.00	&	3	&	8.66	&	3600	&	427371	&	8.08	&	353.10	&	31	&	5.75	&	-0.32	&	2.93	&	2824	&	13022	\\
\midrule
\multirow{3}{*}{60}	&	0.90	&	2	&	0.00	&	290	&	58686	&	3.60	&	-2.67	&	14	&	0.00	&	0.00	&	0.00	&	15	&	84	\\
	&	0.95	&	3	&	1.89	&	2585	&	327116	&	4.63	&	231.29	&	19	&	1.82	&	-0.13	&	0.00	&	184	&	1323	\\
	&	1.00	&	3	&	10.92	&	3600	&	456224	&	8.77	&	468.14	&	46	&	6.51	&	7.81	&	11.90	&	3600	&	12197	\\
\bottomrule
\end{tabular}
\caption{Average results on $30\times 30$ instances.}
\label{t30}
\end{center}
\end{table}


\begin{table}[ht]
\small
\begin{center}
\begin{tabular}{ccrrrrrrrrrrrr}
\toprule
	&		&	\multicolumn{4}{c}{$\ip$}					&\multicolumn{8}{c}{$\ipz$}						\\
				\cmidrule(lr){3-6}							\cmidrule(lr){7-14}
	&		&	\multicolumn{1}{c}{Root}		&\multicolumn{3}{c}{B\&C}	&\multicolumn{3}{c}{Root}	&\multicolumn{5}{c}{B\&C}			\\
			\cmidrule(lr){3-3}		\cmidrule(lr){4-6}				\cmidrule(lr){7-9}		\cmidrule(lr){10-14}
$B$	&	$r$	&	Time	&	Gap	&	Time	&	Nodes	&	$\Delta$LB&	$\Delta$UB&	Time	&	$\Delta$LB&	$\Delta$UB&	Gap	&	Time	&	Nodes	\\
\midrule
\multirow{3}{*}{20}	&	0.90	&	3	&	0.01	&	402	&	39345	&	2.15	&	-2.23	&	9	&	0.00	&	0.00	&	0.01	&	8	&	2	\\
	&	0.95	&	4	&	1.53	&	3303	&	246765	&	3.80	&	-6.49	&	16	&	1.38	&	-0.18	&	0.00	&	20	&	123	\\
	&	1.00	&	5	&	5.27	&	3600	&	257034	&	5.83	&	240.27	&	33	&	4.22	&	-1.14	&	0.15	&	1193	&	5751	\\
\midrule
\multirow{3}{*}{40}	&	0.90	&	3	&	0.57	&	2221	&	195118	&	3.14	&	-4.13	&	21	&	0.53	&	-0.05	&	0.00	&	22	&	73	\\
	&	0.95	&	5	&	4.32	&	3600	&	260138	&	4.98	&	61.81	&	34	&	3.46	&	-1.02	&	0.00	&	175	&	1070	\\
	&	1.00	&	5	&	10.93	&	3600	&	223986	&	8.51	&	408.39	&	66	&	6.64	&	5.31	&	9.65	&	3360	&	7815	\\
\midrule
\multirow{3}{*}{60}	&	0.90	&	3	&	0.57	&	1725	&	150346	&	2.89	&	-3.26	&	25	&	0.51	&	-0.07	&	0.00	&	48	&	308	\\
	&	0.95	&	4	&	4.61	&	3600	&	265815	&	4.73	&	257.09	&	40	&	3.89	&	-0.91	&	0.01	&	1023	&	4430	\\
	&	1.00	&	5	&	13.31	&	3602	&	216712	&	8.68	&	492.87	&	96	&	6.71	&	12.08	&	17.39	&	3600	&	6045	\\
\bottomrule
\end{tabular}
\caption{Average results on $40\times 40$ instances.}
\label{t40}
\end{center}
\end{table}


We observe that the instances become more challenging as $n$, $B$, and $r$ increase, with both formulations taking longer and exploring more nodes. In particular, they become extremely hard for $r=1.00$.

When looking at the results of branch-and-cut, we observe that $\ipz$ clearly outperforms $\ip$ in both time and number of nodes explored, with reductions of up to two orders of magnitude. The sole exception are the hardest instances for which both formulations run out of time. 

Also note that $\ipz$ spends more time processing the root node than $\ip$, yielding improvements of the lower bound of up to 10\%. This improvements seems to increase with $B$ and $r$. The quality of the incumbent found by $\ipz$ at the root decreases with the instance size, but it seems that CPLEX manages to find good solutions further in the search tree. The lower bounds obtained by $\ipz$ are always better than those obtained by $\ip$. The upper bounds are most often better, except for hardest instances.

\textbf{Acknowlegements:} Gustavo Angulo was supported by a Postdoctoral Fellowship in Operation Research at CORE, Universit\'e catholique de Louvain. Mathieu Van Vyve was supported by the Interuniversity Attraction Poles Programme P7/36 COMEX of the Belgian Science Policy Office and the Marie Curie ITN "MINO" from the European Commission.

\bibliographystyle{amsplain} 
\bibliography{fctpdp}

\end{document}